\documentclass[reqno,12pt]{amsart}

\usepackage{amsmath}
\usepackage{amsfonts}
\usepackage{amssymb}
\usepackage{eufrak}
\usepackage{amscd}
\usepackage{amsthm}
\usepackage{epsfig}
\usepackage{amstext}
\usepackage[all]{xy}

\theoremstyle{plain}

 \newtheorem{theorem}{Theorem}

 \newtheorem{cor}[theorem]{Corollary}
 \newtheorem{lemma} [theorem] {Lemma}
 \newtheorem{remark}[theorem]{Remark}

\begin{document}

 \title{A generalized spectral radius formula and Olsen's question}
\author{Terry Loring and Tatiana Shulman}

\address{Department of Mathematics and Statistics, University of New Mexico,
Albuquerque, NM 87131, USA.}

\address{Department of Mathematical Sciences, University of Copenhagen, Universitetsparken 5, DK-2100 Copenhagen, Denmark}

%\email{shulman@math.ku.dk}

\subjclass[2000]{}


\date{\today}

\maketitle

\begin{abstract} Let $A$ be a $C^*$-algebra and $I$ be a closed ideal in $A$.  For $x\in A$, its image under the canonical surjection $A\to A/I$ is denoted by $\dot x$, and the spectral radius of $x$ is denoted by $r(x)$. We prove that $$\max\{r(x), \|\dot x\|\} = \inf \|(1+i)^{-1}x(1+i)\|$$ (where  infimum is taken over all $i\in I$ such that $1+i$ is invertible), which generalizes
spectral radius formula of Murphy and West \cite{MurphyWest} (Rota for $\mathcal{B(H)}$ \cite{Rota}). Moreover if  $r(x)< \|\dot x\|$ then the infimum is attained. A similar result is proved for commuting family of elements
 of a $C^*$-algebra.

Using this we give a partial answer to an open question of C. Olsen: if $p$ is a polynomial then for "almost every'' operator $T\in B(H)$ there is a compact perturbation $T+K$ of $T$ such that $$\|p(T+K)\| = \|p(T)\|_e.$$

We show  also that if operators $A,B$ commute, $A$ is similar to a contraction and $B$ is similar to a strict contraction then they are simultaneously similar to contractions.
\end{abstract}

\section{Introduction}

Let $A$ be a $C^*$-algebra, $x$ its element. The spectral radius of $x$ will be denoted by $r(x)$.
Spectral radius formula of Murphy and West \cite{MurphyWest} (Rota for $\mathcal{B(H)}$ \cite{Rota}) is $$r(x) = \inf \|s^{-1}xs\|,$$ where infimum is taken over all invertible elements in the $C^*$-algebra.  In present paper we get the following generalization of this formula.
Let $I$ be a closed ideal in $A$.  For $x\in A$, by $\dot x$ we denote the image of $x$ under the canonical surjection $A\to A/I$. We prove (Corollary~ \ref{formula}) that $$\max\{r(x), \|\dot x\|\} = \inf \|(1+i)^{-1}x(1+i)\|,$$ where  infimum is taken over all $i\in I$ such that $1+i$ is invertible. In the case $I=A$ it is the spectral radius formula. Moreover we prove that if $\|\dot x\| > r(x)$ then the infimum is attained. For a finite family of commuting elements, the corresponding $i$ can be chosen common (Theorem \ref{formulaComm}).

We apply these results to a question posed by C. Olsen in \cite{OlsenPerturb}. Her question is the following:

\medskip

{\it given an operator $T$ and a polynomial $p$ does there exist a compact perturbation $T+K$ of $T$ such that $\|p(T+K)\| = \|p(T)\|_e$?}

\medskip
\noindent(here $\| \; \|_e$ is the essential norm of an operator). Olsen's second question was if this compact operator $K$ can be chosen simultaneously for all polynomials.

The both questions are open, but there are partial results (\cite{Akemann}, \cite{OlsenStructureTheorem}, \cite{OlsenPlast}, \cite{OlsenPerturb},
 \cite{SmithAndK}, \cite{SmithWard}, \cite{semiproj}), obtained either for special classes of operators or for special
 polynomials. Below we list most of these results:

 \medskip

 1) If $p$ is arbitrary and $T$ is essentially normal or subnormal or n-normal operator or a nilpotent weighted shift, then there is a positive answer to Olsen's question. Moreover $K$ can be chosen independently of polynomial \cite{OlsenPlast}, \cite{OlsenPerturb};

2) If $p(T)$ is compact, then  the answer is positive \cite{OlsenStructureTheorem};

3) If $T$ is arbitrary,  $p$ is linear, the answer is positive. Moreover $K$ can be chosen common for all linear polynomials \cite{SmithAndK};

4) If $p$ is a monomial $p(x) = x^n$, $T$ arbitrary - the answer is positive. Moreover $K$ can be chosen common for finitely many monomials. If $T$ is not quasinilpotent, then $K$ can be chosen common for all monomials  \cite{Akemann}.

\medskip

In present paper  for finitely many arbitrary polynomials $p_1, \ldots, p_n$ we show (Theorems \ref{+} and \ref{big set}) that there is a dense open subset $\Sigma_{p_1, \ldots, p_n}\subset \mathcal{B(H)}$ such that for any operator $T\in \Sigma_{p_1, \ldots, p_n}$, there is $K\in \mathcal{K(H)}$ such that
$$\|p_i(T+K)\| = \|p_i(T)\|_e,$$ $i=1, \ldots, n$.

In other words, those operators for which we cannot answer Olsen's question, form a nowhere dense set.

In \cite{OlsenPlast}, Theorem 5.2, for either $T$ or $T^*$ quasitriangular there was solved an "approximate version'' of the problem: $$\inf_{K\in \mathcal{K(H)}} \|p(T+K)\| = \|p(T)\|_e.$$ We prove that this holds for arbitrary operator $T$ (Theorem \ref{approximate}).

Finally, in Section 3, we present a result which is not formally  related with generalized spectral radius formula, but uses arguments similar to ones used in the proof of the formula. It is an open question whether two commuting operators which are each similar to a contraction are simultaneously similar to contractions (\cite{Davidson}, page 159).
In \cite{FongSourour} a confirmative answer is obtained for the case when the both operators are similar to strict contractions. We show (Theorem \ref{CommonSimilarity}) that if only one of these operators is similar to a strict contraction, then the answer is positive.

The authors are grateful to Victor Shulman for many helpful discussions.

 \section{Generalized spectral radius formula}

  \begin{theorem}\label{formulaComm} Let $A$ be a $C^*$-algebra, $I$ its ideal, and $a_1, \ldots, a_n \in A$ commute. Then for any $\epsilon > 0$, there is $e\in I$ such that $$\|(1+e)a_j(1+e)^{-1}\|- \epsilon \;\le \; \max\{r(a_j), \|\dot a_j\|\}\;\le\;  \|(1+e)a_j(1+e)^{-1}\|, $$ for all  $j$. If $ r(a_j)< \|\dot a_j\|$, for all $j$, then there is $e\in I$ such that $$\max\{r(a_j), \|\dot a_j\|\} =   \|(1+e)a_j(1+e)^{-1}\|,$$ for all  $j$.
  \end{theorem}

  \medskip

  We will need a lemma.

\begin{lemma}\label{lemma} Let $A$ be a  $C^*$-algebra, $I$ its ideal, and $a_1,\ldots, a_n\in A$ commute.  Suppose  $r(a_j)<1$,  $\|\dot a_j \| \le 1$, for all $i$.
Then there exists $e\in I$ such that $$\|(1+e)a_j(1+e)^{-1}\| \le 1,$$ for all $j$.
\end{lemma}
\begin{proof} By \cite{TT} there exists $0\le i_0\le 1$ in $I$ such that $\|(1-i_0)a_j\| \le 1$, $j=1, \ldots, n$.
Let $i = 1 - (1-i_0)^2$. Then $${a_j}^*(1-i)a_j \le 1,$$ $j=1, \ldots, n$.

Since $r(a_j)<1$, by Cauchy's root test \begin{equation}\label{series}\sum_{k=1}^{\infty} \|{a_j}^k\|^2 < \infty.\end{equation}

Consider series $$\sum_{k_1,\ldots,k_n} (a_1^*)^{k_1}\ldots(a_n^*)^{k_n}i\;{a_n}^{k_n}\ldots {a_1}^{k_1}.$$ For its partial sums we have $$\|\sum_{k_1,\ldots,k_n =1}^{N} (a_1^*)^{k_1}\ldots(a_n^*)^{k_n}i\;{a_n}^{k_n}\ldots {a_1}^{k_1}\| \le \sum_{k_1,\ldots,k_n =1}^{N} \|{a_1}^{k_1}\|^2\ldots \|{a_n}^{k_n}\|^2  $$ and by (\ref{series}) the series converges.

 Let $$z= 1 + \sum_{k_1+\ldots +k_n\ge 1} (a_1^*)^{k_1}\ldots(a_n^*)^{k_n}i\;{a_n}^{k_n}\ldots {a_1}^{k_1}.$$

  Then $z\ge 1$ and hence invertible. Let $y = z^{1/2}$. Then $y$ is invertible and of the form $1+e$, $e\in I$. For any $1\le j\le n$ we have
 \begin{multline*} {a_j}^*y^2 a_j = {a_j}^*(1-i)a_j +  \sum_{
  k_1+\ldots +k_n\ge 1} (a_1^*)^{k_1}\ldots({a_j}^*)^{k_j+1}\ldots (a_n^*)^{k_n}i\;{a_n}^{k_n}\ldots {a_j}^{k_j+1}\ldots {a_1}^{k_1} \\ \le 1 + \sum_{k_1+\ldots +k_n\ge 1} (a_1^*)^{k_1}\ldots(a_n^*)^{k_n}i\;{a_n}^{k_n}\ldots {a_1}^{k_1}  = y^2 \end{multline*} and
  $$\|ya_jy^{-1}\|^2 = \|y^{-1}{a_j}^*y^2a_jy^{-1}\| \le 1.$$
\end{proof}

  {\it Proof of Theorem \ref{formulaComm}}.
  For each $j=1, \ldots, n$ take any $\delta_j$ such that $$ \max\{r(a_j), \|\dot a_j\|\} < \delta_j \le \max\{r(a_j), \|\dot a_j\|\} + \epsilon.$$ Applying Lemma \ref{lemma} to the elements $\frac{a_j}{\delta_j}$, we find
  $e\in I$ such that
  $$\|(1+e)\frac{a_j}{\delta_j}(1+e)^{-1}\|\le 1,$$ $j=1, \ldots, n$.   Hence  $$\|(1+e)a_j(1+e)^{-1}\|\le \delta_j \le \epsilon + \max\{r(a_j), \|\dot a_j\|\}$$ $j=1, \ldots, n$. The inequality $$\|(1+e)a_j(1+e)^{-1}\|\ge  \max\{r(a_j), \|\dot a_j\|\}$$ is clear.

   If $\| \dot a_j\| > r(a_j)$, we may take $\delta_j = \| \dot a_j\|$ and obtain the second statement.
        \qed

        \medskip
 In particular, for a single element we get the following formula.

\begin{cor}\label{formula} Let $A$ be a $C^*$-algebra, $I$ its ideal, $a\in A$. Then $$\max\{r(a), \|\dot a\|\} = \inf \|(1+e)a(1+e)^{-1}\|$$ (here infimum in the right-hand side is taken over all $e\in I$ such that $1+e$ is invertible). If $\|\dot a\| > r(a)$ then the infimum is attained.
  \end{cor}

\begin{remark} In the case $r(a)\ge \|\dot a\|$, infimum in the formula need not be attained. First of all, if
$r(a)= \|\dot a\| = 0$, it obviously is not attained.

\noindent If $r(a) \ge \|\dot a\| \neq 0$ it also might not be attained.
For example, consider an operator of the form $$T = \left(\begin{array}{cc}1& K \\ 0& 1 \end{array}\right),$$ where K is compact.
Then $r(T) = \|T\|_e = 1$. If infimum in the formula was attained, $T$ would be similar to a contraction and hence
power-bounded. But   $$\|T^n\| = \|\left(\begin{array}{cc}1& nK\\ 0&1 \end{array}\right)\| \to \infty.$$
\end{remark}

\section{Olsen's question}

Let $\mathcal{B(H)}$ be the space of all bounded operators on a Hilbert space $\mathcal H$, $\mathcal{K(H)}$ the ideal of all compact operators, $Q(\mathcal H) = \mathcal{B(H)}/\mathcal{K(H)}$ Calkin algebra,
 $q: \mathcal{B(H)} \to Q(\mathcal H)$ the canonical surjection.

For $T\in \mathcal{B(H)}$, $\|\dot T\|$ and $r(\dot T)$ are called the essential norm and the essential spectral radius of $T$.  We will use also their standard  denotations $\|T\|_e$ and  $r_e(T)$.

Let $p_1, \ldots, p_n$ be polynomials.
Define  $$G_{p_1, \ldots,  p_n} = \{a\in Q(\mathcal H)\; |\; r(p_i(a))< \|p_i(a)\|, i=1, \ldots, n\}$$ and

$$\Sigma_{p_1, \ldots,  p_n} = q^{-1}(G_{p_1, \ldots, p_n}).$$

\begin{lemma}\label{11}\label{-1} For any polynomials $p_1, \ldots, p_n,$
\[
\Sigma_{p_1, \ldots, p_n} =
\left\{
	T \in \mathcal {B}(\mathcal {H})
	\left|
		\strut\;
		\exists K \in \mathcal {K}(\mathcal{H})
		\mathrm{\ such\ that\ }
		r(p_i(T+K))< \|p_i(T)\|_{e}, \;i=1, \ldots, n
	\right.
\right\} .
\]
\end{lemma}
\begin{proof} Suppose for $T\in \mathcal{B(H)}$ there exists $K\in \mathcal{K(H)}$ such that $r(p_i(T+K))<\|p_i(T)\|_e$, for all $i$. Then  $$r(p_i(\dot T)) \le r(p_i(T+K)) <\|p_i(\dot T)\|,$$ whence $\dot T \in G_{p_1, \ldots, p_n}$ and $T\in \Sigma_{p_1, \ldots, p_n}$.

Suppose $T\in \Sigma_{p_1, \ldots, p_n}$. By \cite{Stampfli} there is $K\in \mathcal{K(H)}$ such that $\sigma(T+K)$ is obtained by filling in some holes in the essential spectrum of $T$. Then by the maximum principle $$r(p_i(T+K)) =\max_{t\in \partial\sigma(T+K)}|p_i(t)| = \max_ {t\in \partial \sigma_e(T)}|p_i(t)|= r(p_i(\dot T)) < \|p_i(T)\|_e,$$ $i = 1, \ldots, n$.
\end{proof}

\begin{theorem}\label{+} For any polynomials $p_1, \ldots, p_n$ and an operator $T\in \Sigma_{p_1, \ldots, p_n}$, Olsen's question has a positive answer. That is there is a compact operator $K$ such that
$\|p_i(T+K)\| = \|p_i(T)\|_e$, $i=1, \ldots, n$.
\end{theorem}
\begin{proof} By Lemma \ref{11} there is  $K_1\in \mathcal{K(H)}$ such that $r(p_i(T+K_1)) < \|p_i(T)\|_e$, $i=1, \ldots, n$. By Theorem~ \ref{formulaComm} there exists $K_2\in \mathcal{K(H)}$ such that for all $i$ $$\|p_i(T)\|_e = \| (1+K_2)^{-1}p_i(T+K_1)(1+K_2)\| = \\
\|p_i((1+K_2)^{-1}(T+K_1)(1+K_2))\|.$$ Let $K = (1+K_2)^{-1}(T+K_1)(1+K_2) - T$. Then $K\in \mathcal{K(H)}$ and $\|p_i(T+K)\| = \|p_i(T)\|_e$ for all $i$.
\end{proof}

Below we present some new examples of operators for which  Olsen's question has a positive answer.

\subsection{Special cases}

\begin{theorem} For any quasinilpotent operator and any polynomial $p$ such that $p(0)=0$, Olsen's question has a positive answer.
\end{theorem}
\begin{proof} Let $T$ be a quasinilpotent operator. If $p(T)$ is compact, the assertion follows from Olsen's structure theorem for polynomially compact operators (\cite{OlsenStructureTheorem}). If $p(T)$ is not compact, then $r(p(T)) < \|p(T)\|_e$ and $T\in \Sigma_p$. The assertion follows now from Theorem~ \ref{+}.
\end{proof}

For nilpotent operators, one does not even need an assumption $p(0)=0$.

 \begin{lemma}\label{triang} Let $T\in \mathcal{B(H)}$ and $(T-t_N)^{k_N}(T-t_{N-1})^{k_{N-1}}\ldots (T-t_1)^{k_1}=0$. Let
 \begin{multline*} \\ \mathcal H_1 = Ker (T-t_1) \\ \mathcal H_2 = Ker (T-t_1)^2 \ominus \mathcal H_1 \\ \ldots \\ \mathcal H_{k_1} = Ker (T-t_1)^{k_1}\ominus \mathcal H_{k_1-1} \\ \mathcal H_{k_1 + 1} =
Ker (T-t_2)(T-t_1)^{k_1} \ominus \mathcal H_{k_1}\\ \ldots \\ \mathcal H_{k_1+\ldots + k_N -1} = Ker (T-t_N)^{k_N -1}(T-t_{N-1})^{k_{N-1}}\ldots (T-t_1)^{k_1} \ominus \mathcal H_{k_1+\ldots + k_N -2}.\end{multline*} Then with respect to the decomposition $\mathcal H = \mathcal H_1 \oplus \ldots \oplus \mathcal H_{k_1+\ldots + k_N -1}$ the operator $T$ has uppertriangular form with $t_1 1, \ldots, t_1 1, \ldots, t_N 1, \ldots, t_N 1$ on the diagonal, where each $t_i 1$ is repeated $k_i$ times.
\end{lemma}
\begin{proof} If $x\in \mathcal H_1$, then $Tx = t_1x$. If $x\in \mathcal H_2$, then $Tx = (T-t_1)x + t_1x$, where $(T-t_1)x \in \mathcal H_1$.
And so on.
\end{proof}

\begin{theorem} For any nilpotent operator $T$ and any polynomial $p$, Olsen's question has a positive answer.
\end{theorem}
\begin{proof} Let $n$ be such that $T^n=0$. By Lemma \ref{triang},
  $T$ is an n-block upper-triangular operator with zeros at the diagonal. Hence $p(T)$ is n-block upper-triangular with $p(0)$ at the diagonal.  Hence $$r(p(T)) = |p(0)|.$$
If $ \|p(T)\|_e > r(p(T))$, then by Theorem \ref{+}, we are done.
So let us assume $$ \|p(T)\|_e = r(p(T))  = |p(0)|.$$ The image $q(p(T))$ of $p(T)$ in Calkin algebra is also n-block upper-triangular  operator with entries in Calkin algebra and with $p(0)$ at the diagonal. Its norm can be equal $|p(0)|$ only if it is diagonal. Thus we have $$p(T) - p(0) 1 \in \mathcal K(\mathcal H).$$ Let $\tilde p = p -p(0)$. Then $\tilde p(T) \in \mathcal K(\mathcal H)$ and by the Olsen's structure theorem for polynomially compact operators (\cite{OlsenStructureTheorem}), there exists $K\in \mathcal K(\mathcal H)$ such that $$\tilde p(T+K) = 0, $$ $$p(T+K) = p(0)1, $$ $$\|p(T+K)\| = |p(0)| = \|p(T)\|_e.$$
\end{proof}

Our purpose now is to show that for any polynomials $p_1, \ldots, p_n$,  the set $\Sigma_{p_1, \ldots, p_n}$ consists of almost all operators.

\subsection{The set $\Sigma_{p_1, \ldots, p_n}$}

\begin{theorem}\label{big set} The complement of $\Sigma_{p_1, \ldots, p_n}$ is nowhere dense.
\end{theorem}

We will need several lemmas.

For any set $S\subseteq \mathbb C$, we denote its boundary by $\partial S$.

 \begin{lemma}\label{semi-Fredholm} Let $T\in \mathcal{B(H)}$ and let $p$ be a polynomial. Then there is $\lambda_0\in \sigma_e(T)$ such that $T-\lambda_0$ is not semi-Fredholm operator and $|p(\lambda_0)| = r_e(p(T)).$
\end{lemma}
\begin{proof} By \cite{Stampfli} there exists $K\in \mathcal{K(H)}$ such that $\sigma(T+K)$ is obtained from $\sigma_e(T)$ by filling in some holes. Hence
\begin{equation}\label{**}\partial \sigma_e(T)\supseteq \partial \sigma(T+K).\end{equation}
By the maximum principle
$$r_e(p(T)) = \max_{t\in \partial \sigma_e(T)}|p(t)| \ge \max_{t\in \partial\sigma(T+K)} |p(t)| = r(p(T+K)).$$
The opposite inequality is obvious. Thus \begin{equation}\label{*} r_e(p(T)) = r(p(T+K)).\end{equation}
By the maximum principle there is a point  $\lambda_0 \in \partial\sigma(T+K)$ such that $$|p(\lambda_0)|  = r(p(T+K)).$$ By (\ref{*}) and (\ref{**}) $\lambda_0 \in \sigma_e(T)$ and $|p(\lambda_0)| = r_e(p(T))$.

Suppose $T-\lambda_0 1$ is semi-Fredholm. Then all operators in some neighborhood of $T - \lambda_0 1$ are semi-Fredholm of the same index \cite{14}. Since $\lambda_0 \in \partial\sigma(T+K)$, in this neighborhood there is an operator $T-\lambda 1$ such that $T+K-\lambda 1$ is invertible. Hence $$ind(T-\lambda_0 1) = ind(T-\lambda 1) = ind (T+K-\lambda 1) = 0.$$ Thus $T-\lambda_0 1$ is semi-Fredholm operator of index $0$, which implies that it is Fredholm. Since $\lambda_0\in \sigma_e(T)$, it is a contradiction.
\end{proof}

The following lemma is inspired by work of Holmes, Scranton and Ward (\cite{HSW}).

 \begin{lemma}\label{density} Let $T\in \mathcal{B(H)}$, $p$ be a polynomial. Then for any $\epsilon > 0$ there is $A\in \mathcal{B(H)}$ such that
 $$\|T-A\|\le \epsilon, \;\; \|p(A)\|_e > r_e(p(A)).$$
 \end{lemma}
 \begin{proof}
 Let $\lambda_0$ be as in Lemma \ref{semi-Fredholm}. Since $T-\lambda_0 1$ is not semi-Fredholm, there is an infinite-dimensional projection $Q$, such that $(T-\lambda_0 1)Q\in \mathcal{K(H)}$ (\cite{Wolf}). Let $\mathcal L = Q\mathcal H$, $\mathcal N = \mathcal L^{\perp}$. With respect to the decomposition  $\mathcal H = \mathcal L \oplus \mathcal N$, we write $$T = \left(\begin{array}{cc} \lambda_0 1+K_{11} &  \ast \\ K_{21}  &T_1\end{array}\right),$$ where $K_{11}, K_{21} \in \mathcal {K(H)}$. For any  operator $S: \mathcal L\to \mathcal L$ we have
 $$T+SQ = \left(\begin{array}{cc}\lambda_0 1+ S +K_{11}&  \ast \\K_{21}&T_1 \end{array}\right),$$
 \begin{equation}\label{1.1}p(T+SQ) =   \left(\begin{array}{cc} p(\lambda_0 1+ S)+ K_{11}' &  \ast \\ K_{21}'&p( {T_1}) + K\end{array}\right),\end{equation} where $K_{11}', K_{21}', K \in \mathcal {K(H)}.$ Hence
 $$\sigma_e(p(T+SQ)) = \sigma_e\left(\left(\begin{array}{cc} p(\lambda_0 1+S)  & \ast \\ 0&p(T_1)\end{array}\right)\right) = \sigma_e(p(\lambda_0 1+S)) \cup \sigma_e(p(T_1)).$$ If  $S$ is quasinilpotent, then $$\sigma_e(p(\lambda_0 1+S)) = \sigma_e(p(\lambda_0)1).$$ Hence \begin{equation}\label{dop}\sigma_e(p(T+SQ)) = \sigma_e(p(T)).\end{equation} By (\ref{dop}) and the choice of $\lambda_0$, \begin{equation}\label{1.2}r_e(p(T+SQ)) = r_e(p(T)) = |p(\lambda_0)|.\end{equation}
So it suffices to construct a nilpotent operator $S$ such that $\|p(T+SQ)\| > |p(\lambda_0)|$
 and $\|S\| \le \epsilon$. Then for $A =T+SQ$ we would have $\|T-A\|\le \epsilon$ and $\|p(A)\|_e>r_e(p(A)).$

Write $p$ in the form $p(t) = a_0 + a_1 t + \ldots + a_n t^n$, where $a_n\neq 0$. Then for any operator $S$, $$p(\lambda_0 +S) = a_0 1 + \mu_1S +\ldots +\mu_nS^n,$$ where \begin{multline*} \mu_1 = a_1+a_2C_2^1\lambda_0+\ldots + a_nC_n^1{\lambda_0}^{n-1}, \\ \mu_2 = a_2+a_3C_3^2\lambda_0 + \ldots + a_nC_n^2{\lambda_0}^{n-1},\\\;\;\;\;\ldots \;\;\;\;\;\; \mu_n = a_n, \end{multline*} and $C_n^k = \frac{n!}{(n-k)!k!}$.
Choose the smallest $i$ such that $\mu_1=0, \ldots, \mu_{i-1} = 0, \mu_i\neq 0$. It exists because $\mu_n\neq 0$. Let $S$ be a block-diagonal operator $$S = \left(\begin{array}{ccc}S'& & \\ & S'&\\& & \ddots\end{array}\right),$$
where $S'$ is $(i+1)$ by $(i+1)$ nilpotent matrix $$S' = \left(\begin{array}{ccccc} 0& & && \\\epsilon & 0 & & &\\ & \epsilon &0 &&\\ & & \ddots & \ddots &\\&&&\epsilon&0\end{array}\right).$$ Then $$\|S\| = \epsilon,$$ $$S^{i+1} = 0, \;\;p(\lambda_0 +S) = p(\lambda_0)1+\mu_iS^i = \left(\begin{array}{ccc}p(\lambda_0) 1 + \mu_iS'^i& & \\ & p(\lambda_0) 1 + \mu_iS'^i&\\& & \ddots\end{array}\right),$$ and  \begin{equation}\label{1.3}\|p(\lambda_0+S)\|_e = \|p(\lambda_0) 1 + \mu_iS'^i\| > |p(\lambda_0)|.\end{equation}
 By (\ref{1.1}), (\ref{1.2}), (\ref{1.3}) we get \begin{multline*}\|p(T+SQ)\|_e = \|\left(\begin{array}{cc} p(\lambda_0 1+S) &  \ast \\ 0&p( {T_1})\end{array}\right)\|_e \ge \|p(\lambda_0 1+S)\|_e   > |p(\lambda_0)|.\end{multline*}
 \end{proof}

%\begin{remark}  In a similar way one can prove that the set of operators $A$ such that
%$\|p(A)\|_e > r_e(p(A))$ for all polynomials $p$ of degree $n$, is dense in $\mathcal{B(H)}$.
%\end{remark}

\begin{cor}\label{easy} $G_p$ is dense in $Q(\mathcal H)$.
\end{cor}

\begin{lemma}\label{0} $G_p$ is open in $Q(\mathcal H)$.
\end{lemma}
\begin{proof} Let $a_n$ belong to the complement of $G_p$, $a_n\to a$. Then $p(a_n)\to p(A)$.
By the semi-continuity of spectral radius, $$r(p(a)) \ge \limsup r(p(a_n)) = \limsup \|p(a_n)\| = \|p(a)\|.$$ The opposite inequality is obvious. Thus the complement of $G_p$ is closed.
\end{proof}

\medskip

\noindent {\it Proof of Theorem \ref{big set}}.
  Lemma \ref{0} implies that $\Sigma_{p_i}$ is open in $\mathcal{B(H)}$, for each $i$.  Let us prove that each $\Sigma_{p_i}$ is dense. Let $B\subset \mathcal{B(H)}$ be an open set. Since $q$ is an open map, it follows from Corollary~ \ref{easy} that
 $q(B)\cap G_{p_i} \neq \emptyset$. Therefore there exist $T\in B$, $A\in \Sigma_{p_i}$ such that $q(T) = q(A)$, whence $T-A\in \mathcal{K(H)}$. Since $\Sigma_{p_i}$ is closed under compact perturbations,
$T\in \Sigma_{p_i}$ and $B\cap \Sigma_{p_i} \neq \emptyset$.
Thus each $\Sigma_{p_i}$ is open and dense and since intersection of finitely many open dense sets is open and dense, we conclude that $$\Sigma_{p_1, \ldots, p_n} = \bigcap_{i=1}^n \Sigma_{p_i}$$ is open and dense. This implies that its complement is nowhere dense. \qed

\medskip

%Thus relations $\|p(x)\|\le C$ are liftable from Calkin algebra for "almost all" elements $x$. Now we show that strict %relations $\|p(x)\|< C$ are liftable for all elements.

Finally we get an approximate version of Olsen's question.

\begin{theorem}\label{approximate} Let $T\in \mathcal{B(H)}$. There exists a sequence $K_n\in \mathcal{K(H)}$ such that for each polynomial $p$, $$\|p(T+K_n)\| \to \|p(T)\|_e.$$
\end{theorem}
\begin{proof} Arguing as in the proof of Lemma \ref{semi-Fredholm}, one find a compact perturbation $T+K$ of $T$, such that for each polynomial $p$,
\begin{equation}\label{=} r(p(T+K)) = r_e(p(T+K)) \le \|p(T+K)\|_e.\end{equation}
Let $\mathcal F = \{p_1, p_2, \ldots\}$ be the set of all polynomials with rational coefficients. Applying Theorem \ref{formulaComm}  to the family $\{p_j(T+K)\}$, $j= 1, \ldots, n$, and using (\ref{=}), we find  $e_n\in \mathcal{K(H)}$ such that $$ \|(1+e_n)p_j(T+K)(1+e_n)^{-1}\| \le \|p_j(T+K)\|_e + \frac{1}{n},$$  $j=1, \ldots, n$ and hence $$ \|p_j\left((1+e_n)(T+K)(1+e_n)^{-1}\right)\| \le \|p_j(T+K)\|_e + \frac{1}{n},$$  $j=1, \ldots, n$.
Let $$K_n = (1+e_n)(T+K)(1+e_n)^{-1} - T.$$ Then $K_n \in \mathcal{K(H)}$. Since for any polynomial $p$, $\|p(T+K)\|_e = \|p(T)\|_e$, we have $$ \|p_j(T+K_n)\| \le \|p_j(T)\|_e + \frac{1}{n}$$ for  $j=1, \ldots, n$.

 Hence for any $p\in \mathcal F$, there is $n_p$ such that for all $n>n_p$,  \begin{equation}\label{100}\|p(T+K_n)\| \le \|p(T)\|_e + \frac{1}{n},\end{equation} and applying this to the polynomial $p(z)=z$, we conclude that the sequence $K_n$ is bounded. Let $C = \sup \|T+K_n\|$. For  arbitrary polynomial $p$ there is $p_{k_n}\in \mathcal F$ such that $$\sup_{|z|\le C} |p(z) - p_{k_n}(z)| < \frac{1}{n}.$$ Now using (\ref{100}) and  von Neumann's inequality, we get for all sufficiently large $n$
\begin{multline*}\|p(T+K_n)\| \le \|p_{k_n}(T+K_n)\| + \|(p-p_{k_n})(T+K_n)\| \le \\ \|p_{k_n}(T+K_n)\| + \sup_{|z| \le C} |p_{k_n}(z)-p(z)| < \|p_{k_n}(T)\|_e + \frac{2}{n} \le \\ \|p(T)\|_e + \sup_{|z| \le \|T\|_e} |p_{k_n}(z)-p(z)| + \frac{2}{n}\; \le \;\|p(T)\|_e + \frac{3}{n}.\end{multline*}  Since $\|p(T+K_n)\|\ge \|p(T)\|_e$, we obtain $$\|p(T+K_n)\| \to \|p(T)\|_e.$$
\end{proof}

\section{Common similarity to contractions}

\begin{lemma}\label{1} Let $A$ be a $C^*$-algebra, $a, b \in A$, $[a, b]=0$, $r(a)<1$, $\|b\|\le 1$. Then  there is an invertible $y\in A$ such that $$\|yay^{-1}\| < 1, \;\; \|yby^{-1}\|\le 1.$$
\end{lemma}
\begin{proof} Let $z = 1+ a^*a + (a^*)^2a^2 + \ldots$ and let $y=z^{1/2}$. By the proof of Lemma \ref{lemma}, applied to the case of $I = A$ and  a single element $a$, $$\|yay^{-1}\| < 1.$$ We have \begin{multline*} b^*y^2b = b^*b + b^*a^*ab + b^*(a^*)^2a^2b + \ldots = b^*b + a^*b^*ba + (a^*)^2b^*ba^2 + \ldots  \le \\ b^*b + a^*a + (a^*)^2a^2 + \ldots = b^*b-1 + y^2 \le y^2\end{multline*} and $$\|yby^{-1}\|^2 = \|y^{-1}b^*y^2by^{-1}\| \le 1.$$
\end{proof}
\begin{theorem}\label{CommonSimilarity} Let $A$ be a $C^*$-algebra, $a, b \in A$,  $[a, b]=0$, $a$ is similar to a strict contraction and $b$ is similar to a contraction. Then there is an invertible $c\in A$ such that $$\|cac^{-1}\| \le 1, \;\; \|cbc^{-1}\|\le 1.$$
\end{theorem}
\begin{proof} Since $b$ is similar to a contraction, there is $s$ such that $$ \|sbs^{-1}\| \le 1.$$ Since $a$ is similar to a strict contraction, $$r(sas^{-1}) = r(a) < 1.$$ Also we have $$[sas^{-1}, sbs^{-1}]=0.$$ By Lemma \ref{1}, there is $y$ such that $$\|ysas^{-1}y^{-1}\|<1, \;\; \|ysbs^{-1}y^{-1}\| \le 1.$$
\end{proof}

 \end{document}